\documentclass[letterpaper, 10 pt, conference]{ieeeconf} 

\IEEEoverridecommandlockouts      
\usepackage{amssymb}
\usepackage{amsmath}
\usepackage{color}
\newtheorem{definition}{Definition}
\newtheorem{assumption}{Assumption}

\newtheorem{theorem}{Theorem}
\newtheorem{remark}{Remark}
\newtheorem{lemma}{Lemma}
\usepackage{balance}
\usepackage{cite}

\usepackage{bm}
\usepackage{graphicx}
\usepackage{footmisc}
\usepackage{geometry}
\usepackage{epstopdf}
\usepackage{algorithm,algpseudocode,float}
\usepackage{lipsum}

\makeatletter
\newenvironment{breakablealgorithm}
{
		\begin{center}
			\refstepcounter{algorithm}
			\hrule height.8pt depth0pt \kern2pt
			\renewcommand{\caption}[2][\relax]{
				{\raggedright\textbf{\ALG@name~\thealgorithm} ##2\par}%
				\ifx\relax##1\relax 
				\addcontentsline{loa}{algorithm}{\protect\numberline{\thealgorithm}##2}%
				\else 
				\addcontentsline{loa}{algorithm}{\protect\numberline{\thealgorithm}##1}%
				\fi
				\kern2pt\hrule\kern2pt
			}
		}{
		\kern2pt\hrule\relax
	\end{center}
}
\makeatother

\title{\LARGE \bf
	Detailed Derivations of ``A Privacy-Preserving  Finite-Time Push-Sum based Gradient Method for Distributed Optimization over Digraphs''}
\author{Xiaomeng Chen$^{1}$, \thanks{$^{1}$X. Chen and L. Shi are with the Department of Electronic and Computer Engineering, Hong Kong University of Science and Technology, Clear Water Bay, Kowloon, Hong Kong (email: xchendu@connect.ust.hk,  eesling@ust.hk).}  
	Wei Jiang$^{2}$, \thanks{$^{2}$W. Jiang e-mail: wjiang.lab@gmail.com }
Themistoklis Charalambous$^{3}$\thanks{$^{3}$T. Charalambous is with the Department of Electrical and Computer Engineering, School of Engineering, University of Cyprus, Nicosia, Cyprus.  He is also a Visiting Professor at Aalto University, Espoo, Finland, and FinEst Centre for Smart Cities, Tallinn, Estonia. (e-mail:
		charalambous.themistoklis@ucy.ac.cy). },  \IEEEmembership{Senior Member, IEEE},
 and Ling Shi$^{1}$, \IEEEmembership{Fellow, IEEE}}

\begin{document}
	\maketitle
	\thispagestyle{empty}
	\pagestyle{empty}
	
	\begin{abstract}
		This paper addresses the problem of distributed optimization, where a network of agents represented as a directed graph (digraph) aims to collaboratively minimize the sum of their individual cost functions. Existing approaches for distributed optimization over digraphs, such as Push-Pull, require agents to exchange explicit state values with their neighbors in order to reach an optimal solution. However, this can result in the disclosure of sensitive and private information. To overcome this issue, we propose a state-decomposition-based  privacy-preserving finite-time push-sum (PrFTPS) algorithm without any global information,  such as network size or graph diameter. Then, based on PrFTPS,  we design a gradient descent algorithm (PrFTPS-GD) to solve the distributed optimization problem.  It is proved that under PrFTPS-GD,  the privacy of each agent is preserved and the linear convergence rate related to the optimization iteration number is achieved.   Finally, numerical simulations are provided to illustrate the effectiveness of the proposed approach.\end{abstract}
	
	\begin{keywords}Distributed optimization, privacy-preserving, finite-time consensus, directed graph. 
		
	\end{keywords}
	\section{INTRODUCTION}
	In this paper, we consider an optimization problem in a multi-agent system of $n$ agents. Each agent $i$ has a private cost function $f_i$, which is  known to itself only. All the  agents aim to collaboratively solve the following optimization problem

	\begin{equation}\label{pro1}	
		\min\limits_{x\in\mathbb{R}^{p}}F(x) :=\sum_{i=1}^{n}f_{i}(x),
	\end{equation}
	where $x$  is the global decision variable. The agents are connected through a communication graph and can only transmit messages to their neighbors. By local computation and communication, each agent seeks a solution that minimizes the sum of all the local objective functions. Such a distributed paradigm facilitates breaking large-scale problems into  sequences of smaller ones. That is why it has  been widely adopted in several applications, such as  power grids~\cite{braun2016distributed}, sensor networks \cite{dougherty2016extremum} and vehicular networks \cite{mohebifard2018distributed}.
	
	{\color{black}To solve problem \eqref{pro1}, decentralized gradient descent (DGD)  is the most commonly used algorithm, requiring diminishing stepsizes to ensure optimality \cite{nedic2009distributed}. To overcome the challenge of slow convergence caused by diminishing stepsizes, Xu et al.\cite{xu2015augmented} adopted the dynamic average consensus \cite{kia2019tutorial} to propose a gradient tracking (GT) method with a constant stepsize. Recently, Xin  et al. \cite{xin2018linear} and Pu et al. \cite{pu2020push}  devised a modified GT algorithm called AB/Push-Pull algorithms for  distributed optimization, which can be applied to a general digraph.  {\color{black}A comprehensive survey on distributed optimization algorithms is provided by Yang et al. \cite{yang2019survey}. }}
	\\ \indent The aforementioned distributed algorithms share state values in each iteration, which can compromise the privacy of   agents if they have private information. By hacking into communication links, an adversary could potentially  access  to transmitted  messages among agents and potentially gather  private information  using an inferring algorithm. Mandal\cite{mandal2016privacy} presented theoretical analysis of privacy disclosure in distributed optimization, where the parameters of cost functions and generation power can be correctly inferred by an adversary. As the number of privacy leakage events is increasing, there is an urgent need to preserve privacy of each agent in distributed systems. 
	
	Recently, many results have been reported on the topic of privacy-preserving distributed optimization. One commonly used approach is differential privacy (DP) \cite{dwork2006calibrating} due to its rigorous mathematical framework, proven privacy preservation properties and ease of  implementation \cite{nozari2017differentially}. 
	However, DP-based approaches face a fundamental trade-off between privacy and accuracy, which may result in suboptimal solutions \cite{huang2015differentially}. To address this challenge, Lu et al. \cite{lu2018privacy} combined distributed optimization methods with partially homomorphic encryption. Nonetheless, this approach has limitations due to high computation complexity and communication costs. {\color{black}To overcome these limitations and achieve accurate results, Wang \cite{wang2019privacy} proposed a privacy-preserving average consensus using a state decomposition mechanism that divides the state of a node into two sub-states. }
	
It is worth noting that none of the aforementioned approaches  is suitable  for agents over digraphs. {\color{black}To preserve privacy of nodes interacting on a digraph, Charalambous et al. \cite{charalambous2019privacy} proposed an offset-adding privacy-preserving push-sum, and Gao et al. \cite{gao2018privacy} protected privacy by adding randomness on edge weights, both of which are only effective against honest-but-curious nodes (see Definition \ref{df3}). To improve resilience to external eavesdroppers (see Definition \ref{df2}), Chen et al. \cite{chen2020privacy} extended the state decomposition mechanism to digraphs and introduced an uncertainty-based privacy notion. In terms of   privacy-preserving distributed optimization   over digraphs, Mao  et al. \cite{mao2020privacy} designed a privacy-preserving algorithm based on the push-gradient method with a decaying stepsize, which lacked a formal privacy notion.}  Wang and Nedi{\'c}  \cite{wang2022tailoring} designed a DP-oriented gradient tracking based algorithm (DPGT) that can ensure both differential privacy and optimality. However, it adopted a diminishing stepsize to ensure convergence, resulting in a slow convergence rate.
	To speed up the convergence, Chen et al. \cite{chen2021differential} proposed a state-decomposition-based push-pull (SD-Push-Pull) algorithm, which guarantees both linear convergence and differential privacy for digraphs.  Nevertheless, SD-Push-Pull only converges to a suboptimal value.
	
	\indent Inspired by  recent results that privacy can be enabled
	in consensus over digraphs by state decomposition \cite{chen2020privacy} and that finite-time push-sum can be used in distributed optimization to deliver the optimal solution \cite{jiang2022fast}, {\color{black}this paper presents a novel PrFTPS algorithm that accurately computes the average value for digraphs in a finite time, as opposed to the asymptotic average consensus achieved in \cite{chen2020privacy}. Then, combined with gradient decent,    PrFTPS-GD is proposed to solve problem \eqref{pro1} allowing each node in a digraph to achieve optimal value linearly while preserving its privacy. The main contributions of this paper are summarized as follows:
	
	\begin{enumerate}
	{\color{black}	
		\item We propose PrFTPS (\textbf{Algorithm \ref{alg1}}) based GD algorithm (\textbf{Algorithm \ref{alg2}})  to solve problem \eqref{pro1} over digraphs. Moreover, we show that  Algorithm \ref{alg1} can compute the exact average value  in finite time and  Algorithm \ref{alg2} guarantees the linear convergence to the optimal value of problem \eqref{pro1} (\textbf{Theorem 1}). 
		\item We analyze the privacy-preserving performance of PrFTPS-GD against  honest-but-curious nodes and eavesdroppers (\textbf{Theorem 2}). Specifically, we adopt the uncertainty-based privacy notion \cite{chen2020privacy} and show that the adversary has infinite uncertainty about  agents' private information under certain topological conditions. 
		\item PrFTPS-GD performance is evaluated via 
simulations and compared with   other state-of-the-art privacy-preserving approaches  (e.g., \cite{wang2022tailoring,chen2021differential}) over digraphs.
It is shown that our approach apart from adopting an easily tuned constant stepsize (unlike the diminishing stepsize in \cite{wang2022tailoring}), it computes the optimal solution instead of the suboptimal one in \cite{chen2021differential}. }      
	\end{enumerate}
	
	{\color{black}
	\textit{Notations:}  In this paper,   $\mathbb{R}^n$ and $\mathbb{R}^{n\times p}$ represent the set of $n$ dimensional vectors and $n\times p$ dimensional matrices. {\color{black}$\mathbb{Z}_{++}$ denotes the set of positive integers.} $\mathbf{1}_n\in\mathbb{R}^n$, $\mathbf{I}_n\in\mathbb{R}^{n\times n}$ and  $\mathbf{0}_n\in\mathbb{R}^{n\times n}$ represent the vector of ones,  the identity matrix and the zero matrix, respectively. For an arbitrary vector $\bm{x},$ we denote its $i$th element by $\bm{x}_i$. For an arbitrary matrix $\bm{M},$ we denote its  element in the $i$th row and $j$th column by $[\bm{M}]_{ij}$. $\otimes$ denotes the Kronecker product. The spectral radius of matrix $\mathbf{A}$ is denoted by $\rho(\mathbf{A})$. Matrix  $\mathbf{A}$ is called row-stochastic if the sum of each row equals to $1$ and the entries of $\mathbf{A}$ are non-negative. Similarly, matrix  $\mathbf{A}$ is called column-stochastic if the sum of each column equals to $1$ and the entries of $\mathbf{A}$ are non-negative.
	\section{PRELIMINARIES AND PROBLEM STATEMENT}
	\subsection{Network Model}
	We consider a digraph $\mathcal{G}\triangleq (\mathcal{V}, \mathcal{E})$ with $n$ nodes, where  the set of nodes and edges are  $\mathcal{V}=\{1,\ldots,n\}$ and $\mathcal{E}\subset \mathcal{V} \times \mathcal{V}$, respectively. A communication link from node $i$ to node $j$ is denoted by $\varepsilon_{ji}=(j,i)\in  \mathcal{E}$, indicating that node $i$ can send messages to node $j$. The nodes who can  send messages to node $i$ are denoted as in-neighbours of node $i$ and the set of these nodes is denoted as $N_i^{-}=\{j\in \mathcal{V}\mid \varepsilon_{ij}\in \mathcal{E}\}$. Similarly, the nodes who can  receive messages from node $i$ are denoted as out-neighbours of node $i$ and the set of these nodes is denoted as $N_i^{+}=\{j\in \mathcal{V}\mid \varepsilon_{ji}\in \mathcal{E}\}$. The
	cardinality of $N_j^{+}$ , is called the out-degree of node $j$ and is denoted as $\mathcal{D}^+_j=|N_j^{+}|$. A digraph is called strongly connected if there exists at least one directed path from any node $i$ to any node $j$ with $i\neq j$.

	\subsection{Push-Sum Algorithm}
	
	The push-sum algorithm, introduced originally in \cite{kempe2003gossip}, aims at achieving average consensus for each node communicating over a digraph which satisfies the following assumption. 
			 \begin{assumption}\label{asp1}
			The digraph $\mathcal{G}$ is assumed to be strongly connected. 		
	\end{assumption}
	
	Consider a network of $n$ nodes, where each node has a private initial state, {\color{black}termed as  $x_i(0)$.} The push-sum algorithm introduces \textit{two auxiliary varaibles}, $x_{i,1}(k)$ and $x_{i,2}(k)$, and assumes the out-degree is known for each node. The details are as follows: for each node $i$, 	$$
	x_{i,l}(k+1)=\sum\limits_{j\in N^{-}_i\cup\{i\}}p_{ij}(k)x_{j,l}(k),\qquad  k\ge 0,    
	l=1,2,
	$$
	where $p_{ij}(k)=1/(1+\mathcal{D}^+_j), \forall i \in N_j^{+} \cup \{j\}$ and $x_{i,1}(0)=x_i(0), x_{i,2}(0)=1$ for $i \in \mathcal{V}$. \\
	\textbf{Proposition 1.} \cite{kempe2003gossip} \textit{If a digraph $\mathcal{G}(\mathcal{V}, \mathcal{E})$ with $n$ nodes satisfies Assumption \ref{asp1}, then  the ratio $r_i(k):=x_{i,1}(k+1)/ x_{i,2}(k+1)$ asymptotically converges to the average of the initial values, i.e.,  we have}
$$
		\lim\limits_{k \rightarrow \infty} r_i(k)=\frac{\sum_{i\in \mathcal{V}}x_i(0)}{n}, \forall i \in \mathcal{V}. 
$$

	\subsection{Information Set and Privacy Inferring Model}}
	{\color{black}Before defining privacy, we first introduce the privacy inferring model. The adversary set $\mathcal{A}$ is assumed  to obtain some online data by eavesdropping on some edges $\mathcal{E}_r \subseteq \mathcal{E}$ and nodes $\mathcal{V}_a  \subseteq \mathcal{V}$. 	The information set accessible to  $\mathcal{A}$ at time $k$  is denoted as $\mathcal{I}_\mathcal{A}(k)$, which contains all transmitted information accessible to   $\mathcal{A}$.
	
		Then all the information accessible to $\mathcal{A}$ at  time iteration $K$ is denoted as $\mathcal{I}_\mathcal{A}(0:K)\triangleq\{\mathcal{I}_\mathcal{A}(0),\mathcal{I}_\mathcal{A}(1),\ldots,\mathcal{I}_\mathcal{A}(K)\}$. 
		
		With the above  model, we adopt an uncertainty-based notion of privacy, which is proposed in \cite{chen2020privacy}. Denote the private information of node $i$ as $x_{p,i} \in \mathbb{R}^p$  and define a set $\Delta\mathcal{I}_\mathcal{A}({\color{black}x_{p,i}})$ as
		$$\Delta\mathcal{I}_\mathcal{A}({\color{black}x_{p,i} })=\{\bar x_{p,i} \mid \text{the adversary's information set } \hspace*{-1mm}= \mathcal{I}_\mathcal{A}(0:K)\},
		$$
		which contains all possible states that can correspond to $x_{p,i}$ when
		the information set accessible to $\mathcal{A}$ is $\mathcal{I}_\mathcal{A}(0:K)$. 
		
		The diameter of $\Delta\mathcal{I}_\mathcal{A}(x_{p,i})$ is defined as 
		$$
			\text{Diam}\{\Delta\mathcal{I}_\mathcal{A}(x_{p,i})\}=\sup\limits_{\bar x_{p,i}, \bar x_{p,i}'\in\ \Delta\mathcal{I}_\mathcal{A}(x_{p,i})} |\bar x_{p,i}-\bar x_{p,i}'|,
		$$
		where $\bar x_{p,i}$ and $\bar x_{p,i}'$ are two different states that belong to set $\Delta\mathcal{I}_\mathcal{A}(x_{p,i})$.

		{   \begin{definition}\label{dfprivacy}
				{\color{black}The privacy of $x_{p,i}$} is preserved against  $\mathcal{A}$ if $\text{Diam}\{\Delta\mathcal{I}_\mathcal{A}(x_{p,i})\}=\infty$.
		\end{definition}}

		In this paper, we consider distributed optimization problems, where local objective gradients usually carry sensitive information. For example, in distribited-optimization-based  localization and rendezvous, directly exchanging the gradient of an agent leads to disclosing its  position \cite{huang2015differentially}. Recent work shows that gradients are directly calculated from and embed sensitive information of training  learning data \cite{zhu2019deep}. {\color{black}Hence, the private information is the \textit{gradient of each agent at all time iteration.} Then, we define the privacy preservation of each agent as follow.
		
		\begin{definition}\label{pr}
		For a network of $n$ agents in distributed
optimization, the privacy of agent $j$ is preserved against $\mathcal{A}$ if the privacy of  its gradient value $\nabla f_j(x_j)$  evaluated at any point $x_j$ is preserved.
		\end{definition} }
		

		We consider two types of adversaries, defined as follows. 
		\begin{definition}\label{df3}
			An  honest-but-curious adversary is a node or a group of nodes which knows the network topology and follows the system's  protocol, attempting to infer the private information of other nodes.
					\end{definition}
		\begin{definition}
			\label{df2}
			An  eavesdropper is an external adversary who has the knowledge of network topology, and is able to eavesdrop on a portion of coupling weights and transmitted data.\end{definition}
		
		
		
\section{Main results}
		
		In this section,  we first  propose a  privacy-preserving finite-time push-sum algorithm via a state decomposition mechanism.   Then, we adopt the proposed  privacy-preserving approach to address  problem \eqref{pro1} based on a gradient descent approach.
				\subsection{Privacy-Preserving Finite-Time Push-Sum Algorithm}

		The main idea of our privacy-preserving approach is a state
decomposition mechanism.

		\textit{Decomposition Mechanism:} Let each node decompose its state $x_{i,l}(k)$ into two substates $x^\alpha_{i,l}(k)$ and $x^\beta_{i,l}(k)$, $l=1,2$. The initial values $x^\alpha_{i,l}(0)$ and $x^\beta_{i,l}(0)$ can be randomly chosen from the set of all real numbers under the following constraint 
		\begin{equation*}
				x_{i,1}^\beta(0)+x_{i,1}^\alpha(0)=2x_{i}(0), x_{i,2}^\alpha(0)=0, x_{i,2}^\beta(0)=2, \forall i \in \mathcal{V},
		\end{equation*}
		where $x_i(0)$ denotes the private initial state of node $i$.

		Under the state decomposition mechanism, the overall dynamics become
		\begin{equation}
			\label{e1}
			\left\{  
			\begin{aligned}
				&x_{i,l}^\alpha(k+1)=\sum\limits_{j\in N^{-}_i\cup\{i\}}p_{ij}(k) x^\alpha_{j,l}(k)+ a_i^{\alpha,\beta}(k)x^\beta_{i,l}(k),\\
				& x^\beta_{i,l}(k+1)=a_i^{\beta,\alpha}(k)x^\alpha _{i,l}(k)+a_i^{\beta,\beta}(k)x^\beta_{i,l}(k),   
			\end{aligned}
			\right. 
		\end{equation}with $i\in \mathcal{V}$, $l=1,2$. In this decomposition scheme, the substate  $x^\alpha_{i,l}(k)$ is exchanged with other nodes while $x^\beta_{i,l}(k)$ is never shared with other nodes. The coupling weights between the two substates $x^\alpha_{i,l}(k)$ and $x^\beta_{i,l}(k)$ are asymmetric and denoted as $a_i^{\alpha,\beta}(k)$ and $a_i^{\beta,\alpha}(k)$. The update weights for substate $x^\beta_{i,l}(k)$ is denoted as  $a_i^{\beta,\beta}(k)$. The outgoing link weight from agent $i$ to agent $j$ is denoted as $p_{ji}(k)$.  
		These are  design parameters and will be designed  in the following weight mechanism (Section \ref{wm}).
		
		We next introduce  details of the weight mechanism to enable  algorithm convergence and privacy preservation.
		\vspace*{2mm}
		\subsubsection{Weight mechanism}\label{wm} For $k=0,\forall i \in \mathcal{V}$,  we set $a_i^{\beta,\beta}(0)=0, a_i^{\alpha,\beta}(0)=1$ and $ p_{ji}(0)=0,  \forall j \notin N_i^{+}$.  Also, we allow  $p_{ji}(0), \forall j \in N_i^{+}\cup \{i\}$ and
		$a_i^{\beta,\alpha}(0)$ to be arbitrarily chosen from the set of all real numbers under the constraint $$\sum_{j=1}^n p_{ji}(0)+a_i^{\beta,\alpha}(0)=1.$$ For $k\ge 1$, we let  $p_{ji}(k)= 1(2+\mathcal{D}_i^+ )$ for $j \in   N_i^{+}\cup \{i\}$ and $p_{ji}(k)= 0$,  otherwise. Also,  $$a_i^{\beta,\alpha}(k)= \frac{1}{2+\mathcal{D}_i^+ },  \quad a_i^{\beta,\beta}(k)=a_i^{\alpha,\beta}(k)= \frac{1}{2}.$$
	\begin{remark}
			Under the above weight mechanism, the state-decomposition-based push-sum \eqref{e1} still preserves the property of conventional push-sum. 
		Rigorous theoretical  analysis will be provided in Section \ref{analysis}.
	\end{remark}

		To obtain the exact average value in finite time, 		we use the minimal polynomial associated with  iteration \eqref{e1}, in conjunction with the final value theorem \cite{charalambous2015distributed, yuan2009decentralised}. Next, we provide definitions on minimal polynomials.
		
		\begin{definition}
			\textit{(Minimal Polynomial of a Matrix.)} The minimal polynomial of matrix $P$, denoted by
			$$Q(t)=t^{D+1}+\sum\limits_{i=0}^D\alpha_i t^i,
			$$ is the monic polynomial of minimum degree $D+1$ that satisfies $Q(P)=0_{n}$ and $\alpha_i$ is the polynomial coefficient. 
			\end{definition}
		\begin{definition}
			\textit{(Minimal Polynomial of a Matrix Pair.)} The minimal polynomial associated with   $[P, e^\top_
			j ]$, denoted by $$Q_j(t) = t^{D_j+1} + \sum_{i=0}^{D_j} \alpha_{j,i}t^i=0, \alpha_{j,i} \in \mathbb{R},$$ is the
			monic polynomial of minimum degree $D_j + 1$ that satisfies
			$e^\top_jQ_j(P)=0$. 
		\end{definition}
		
		In what follows, we will show how to use the coefficients of the minimal polynomial to obtain the final value in finite time. By using the iteration in \eqref{e1}, we have
	$$
			\sum_{i=0}^{D_j+1} \alpha_{j,i}x_{j,1}^\alpha(k+i)=0, \forall k \in \mathbb{Z}_{++},
		$$
		where $\alpha_{j,D_j+1}=1$. Thus,  the minimal polynomial of a matrix is unique due to the monic property. We denote the $\mathcal{Z}$-transform of $x_{j,1}(k)$ as $X_{j,1}=\mathcal{Z}(x_{j,1}(k))$. By the time-shift property of the $\mathcal{Z}$-transform, it is easy to obtain that	
	$$
			X_{j,1}(z)=\frac{{\sum_{i=1}^{D_j+1} \alpha_{j,i}\color{black}\sum_{l=1}^{i}}x_{j,1}^\alpha(l)z^{i-l}}{Q_j(z)}. 
		$$
		Since the communication topology of the networked system is strongly connected, the minimal polynomial $Q_j(z)$ does not have any unstable poles apart from one. Hence, we can define  polynomial
			$$p_j(z)\triangleq\frac{Q_j(z)}{z-1}\triangleq \sum\limits_{i=0}^{D_j}\beta_i^{(j)}z^i. $$
	
		By the final value theorem \cite{charalambous2015distributed} and \cite{yuan2009decentralised},  the final state
		values of \eqref{e1} are computed as
		\begin{equation*}
			\begin{aligned}
				&\phi^\alpha_{x_l}(j)=\lim\limits_{k\rightarrow\infty} x_{j,l}^\alpha(k)=\lim\limits_{z\rightarrow 1}(z-1)X_{j,l}^\alpha(z)=\frac{(x_{l,{D_j}}^\alpha)^\top \bm{\beta}_j}{ \bm{1}^\top\bm{\beta}_j},\\
							\end{aligned}
		\end{equation*}
		where 
		$$(x_{l,{D_j}}^\alpha)^\top=(x_{j,l}^\alpha(1),x_{j,l}^\alpha(1),\ldots, x_{j,l}^\alpha(D_j+1)), l=1,2,
		$$ and $\bm{\beta}_j$ is the coefficient vector of the polynomial $p_j(z)$.
		
		Denote the following vectors of $2k+1$ successive discrete-time values for the two iterations $x_{j,l}^\alpha(k), l=1,2$ at node $j$ as

	{\color{black}$$(x_{l,{2k}}^\alpha)^\top=(x_{j,l}^\alpha(1),x_{j,l}^\alpha(1),\ldots, x_{j,l}^\alpha(2k+1)), l=1,2.		$$}
		
		Moreover, define the associated Hankel matrix and the  difference vectors between successive values for $l=1,2$ as 
	{\color{black}	$$
	\begin{aligned}
			\Gamma\{(x_{l,{2k}}^\alpha)^\top\}&\triangleq\\
			&\begin{bmatrix}
			x_{j,l}^\alpha(1)&x_{j,l}^\alpha(2)&\cdots &x_{j,l}^\alpha(k+1)\\
			x_{j,l}^\alpha(2)&x_{j,l}^\alpha(3)&\cdots &x_{j,l}^\alpha(k+2)\\
			\vdots&\vdots&\ddots&\vdots\\
			x_{j,l}^\alpha(k+1)&x_{j,l}^\alpha(k+2)&\cdots &x_{j,lt}^\alpha(2k+1)\\

		\end{bmatrix},
	\end{aligned}
		$$
					$$
		\begin{aligned}
			(\bar x_{l,{2k}}^\alpha)^\top &\triangleq\\
			&(x_{j,l}^\alpha(2)-x_{j,l}^\alpha(1),\ldots, x_{j,l}^\alpha(2k+2)-x_{j,l}^\alpha(2k+1)). \\
		\end{aligned}
		$$	}
		
		It is shown in \cite{yuan2009decentralised} that for arbitrary initial conditions
		$x_{j,1}^\alpha(1)$ and $x_{j,2}^\alpha(1)$, $\bm{\beta_j}$ can be computed as the kernel of the first defective Hankel matrices $\Gamma\{(\bar x_{1,{2k}}^\alpha)^\top\}$ and $\Gamma\{(\bar x_{2,{2k}}^\alpha)^\top\}$, except a set of initial conditions with Lebesgue measure zero.
		
		{\color{black}From the above analysis, we know $\bm{\beta_j}$ and $D_j$ can be different for  node $j$. Thus, in existing works \cite{jiang2022fast},\cite{charalambous2015distributed}, all nodes  are assumed to know the upper bound of the network size. To relax this  assumption,  Charalambous and Hadjicostis \cite{charalambous2018stop} proposed a distributed termination mechanism, allowing all nodes to agree when to terminate their iterations, given they have all computed the average.  
		The procedure is as follows: 
		\begin{itemize}
			\item Once iterations \eqref{e1} are initiated, each node $j$ also initiates two counters $c_j$, $c_j(0)=0$, and $r_j$, $r_j(0)=0$. Counter $c_j$ increments by one at every time step, i.e., $c_j(k+1)= c_j(k)+1$. The way counter $r_j$ updates is described next.
			\item Alongside iterations \eqref{e1} a $\max$-consensus algorithm is initiated as well, given by
			\begin{align}\label{eq:maxconsensus}
			\theta_j(k+1) = \max_{v_i \in \mathcal{N}_j \cup \{v_j\}}\big\{ \max\{\theta_i(k),c_i(k)\} \big\}, 
			\end{align}
			with $\theta_j(0)=0$. Then, $r_j$ is updated as follows:
			\begin{align}\label{eq:rj}			r_j(k+1)=
			\begin{cases}
			0, & \text{if } \theta_j(k+1) \neq \theta_j(k), \\
			r_j(k)+1, & \text{otherwise}.
			\end{cases}
			\end{align}
			\item Once the  Hankel matrices $\Gamma\{(\bar x_{1,{D_j}}^\alpha)^\top\}$ and $\Gamma\{(\bar x_{2,{D_j}}^\alpha)^\top\}$ lose rank, node $j$ saves the count of the counter $c_j$ at that time step, denoted by $k^o_j$, as $c^o_j$, i.e., $c^o_j\triangleq c_j[k^o_j]$, and it stops incrementing the counter, i.e., $\forall k'\geq k^o_j, c[k']=c_j[k^o_j]=c^o_j$. Note that {\color{black}$c^o_j=2(D_j+1)+1$}. 
			\item Node $j$ can terminate iterations \eqref{e1}  when $r_j$ reaches $c^o_j$.
		\end{itemize}

		Therefore, based on the distributed termination mechanism \cite{charalambous2018stop,jiang2021fully}, we design a  privacy-preserving finite-time push-sum algorithm (PrFTPS) as presented in Algorithm \ref{alg1}, which guarantees the minimum number of iteration steps to obtain the exact average without any global information.	
		\begin{breakablealgorithm}
			\caption{A Privacy-Preserving Finte-Time  Push-Sum Algorithm (PrFTPS)}
			\label{alg1}
			\begin{algorithmic}[1]\\
				{\bf Input:} %
				Initial state $x_j(0)$,  step $t$, graph $\mathcal{G}(\mathcal{V}, \mathcal{E})$.
				\If {$t=0$}
				\\
		Run the privacy-preserving iteration \eqref{e1} and the max-consensus algorithm \eqref{eq:maxconsensus},  store the vectors  $(\bar x_{1,{D_j}}^\alpha)^\top$,  $(\bar x_{2,{D_j}}^\alpha)^\top$, increment the value of the counter $c_j (k)$ and find the value of the counter $r_j(k)$ via \eqref{eq:rj}.\\
		Increase the dimension $k$ of the Hankel matrices $\Gamma\{(\bar x_{1,{D_j}}^\alpha)^\top\}$ and $\Gamma\{(\bar x_{2,{D_j}}^\alpha)^\top\}$ until $k_j^o$ at which they lose rank. Once this happens, store the kernel $\bm{\beta}_j$	of the 	first defective 	matrix and the value {\color{black}$c^o_j=2(D_j+1)+1$}.\\
		Continue  iteration \eqref{e1} until iteration $k_{j,t}$ where 	$r_j(k_{j,t})=c^o_j$ and store $$D_{\text{max}}=\frac{k_{j,t}-2D_j-2}{2}-1.$$ 							
				\Else \\		
	Run the privacy-preserving algoirthm \eqref{e1} for {\color{black}$k_{\max}=D_{\max}+2$} steps with the same  $\bm{\beta_j}$.
				\EndIf
							\State  Compute the average value as $\hat x_j^{ave}=\frac{(x_{1,{D_j}}^\alpha)^\top \bm{\beta}_j}{(x_{2,{D_j}}^\alpha)^\top \bm{\beta}_j}$.
				\State {\bf Output:} Node $j\in\mathcal{V}$ outputs  $\hat x_j^{ave}$.
			\end{algorithmic}
		\end{breakablealgorithm}}
		
		{\color{black}\begin{remark}
			Compared to existing state-decomposition-based privacy-preserving average consensus in \cite{wang2019privacy} and \cite{chen2020privacy}, PrFTPS is applicable to general digraphs, while the method in \cite{wang2019privacy} is limited to undirected graphs with doubly-stochastic methods. Moreover, our innovative weight mechanism (Section \ref{wm}) maintains constant weights in the privacy-preserving iteration \eqref{e1} for $k \ge 1$, as opposed to the time-varying weights in \cite{wang2019privacy} and \cite{chen2020privacy}. These constant weights play a crucial role in the final value theorem \cite{yuan2009decentralised}, allowing PrFTPS to compute an exact average consensus in finite time using the coefficients of the minimal polynomial associated with iteration \eqref{e1}. In contrast, the weight mechanisms in \cite{wang2019privacy} and \cite{chen2020privacy} only permit asymptotic average consensus, which limits their application to solving distributed optimization problems. Our proposed weight mechanism overcomes this limitation and facilitates the application of our PrFTPS algorithm to solve distributed optimization problems while preserving privacy, as shown in Algorithm \ref{alg2}.					\end{remark}}

		\subsection{Finite-Time  Privacy-Preserving Push-Sum based  Gradient Descent Algorithm}
		In this subsection, we design a  PrFTPS based gradient method to  address   problem \eqref{pro1}.  We first assume the following  conditions about Problem \eqref{pro1}. 
	\begin{assumption}\label{asp3}
Each objective function $f_i$ is $\mu-$strongly convex with $L-$Lipschitz continuous gradients, i.e., 		
{\color{black} 
$$\begin{aligned} 
	&\langle\nabla f_i(\mathbf{x})- \nabla f_i(\mathbf{y}), \mathbf{x}-\mathbf{y}\rangle \geq \mu||\mathbf{x}-\mathbf{y}||^2,\\
	&||\nabla f_i(\mathbf{x})- \nabla f_i(\mathbf{y})||\leq L||\mathbf{x}-\mathbf{y}||, \quad \forall \mathbf{x}, \mathbf{y}\in\mathbb {R}^p.\\
\end{aligned}
$$}	
\end{assumption}

Under Assumption \ref{asp3}, Problem \eqref{pro1} has a unique optimal solution $x^{\star}\in\mathbb{R}^{p}$ \cite{pu2020push}.

		To address problem \eqref{pro1} distributively, we propose the following  PrFTPS based GD algorithm inspired by the distributed structure in \cite{xin2018linear, pu2020push,jiang2022fast}. Starting from the initial condition $x_i(0)\in\mathbb{R}^p$ and $y_i(0)=\nabla f_i(x_i(0))$, for all $t\ge 0$, we have	
		\begin{subequations}\label{eqmain}
			\begin{align}	
							& y_i(t) \leftarrow \text{Algorithm \ref{alg1}} (\nabla f_i(x_i(t)),t),\\
				&x_i(t+1)=\sum\limits_{j \in N_i^-\cup\{i\}}\bar a_{ij}x_j(t)-\eta y_i(t),\label{eqa}
			\end{align}	
		\end{subequations}where {\color{black}$\eta$} is the stepsize and $\mathbf{\bar A}=[\bar a_{ji}]\in \mathbb{R}^{n}$ is   row-stochastic. The details are summarized in Algorithm \ref{alg2} in the following. 
		\begin{breakablealgorithm}
			\caption{A Privacy-Preserving Finte-Time  Push-Sum based GD Algorithm (PrFTPS-GD)}
			\label{alg2}
			\begin{algorithmic}[1]\\
				{\bf Initialization: } \hspace*{-1mm}Stepsize $ \eta$, maximum optimization iteration number $T$, graph $\mathcal{G}(\mathcal{V}, \mathcal{E})$.\\
				{\bf Input:} %
				Node $i\in\mathcal{V}$ sets the initial value $x_i(0), y_i(0)=\nabla f_i(x_i(0))$ and $t=0$.
				\For {$t\le T $}
								\State Put $\nabla f_i(x_i(t)), t$ as input to Algorithm \ref{alg1} and get output \hspace*{4mm} $\hat x^{ave}$; design $y_i(t)=\hat x^{ave}$. 
				\State Compute $x_i(t+1)$ using \eqref{eqa} with $y_i(t)$. 
				\State $t\leftarrow t+1$
				\EndFor
				\State {\bf Output:} Node $i\in\mathcal{V}$ obtains the solution $x^\star$.
			\end{algorithmic}
		\end{breakablealgorithm}
		
Algorithm \ref{alg2} guarantees that the number of iterations needed at every optimization step $t\ge 1$ is the minimum. Fig. \ref{ft} shows the number of iterations
		needed at every optimization step.
			\begin{figure}[htp]
		\centering
		\includegraphics[width=0.48\textwidth]{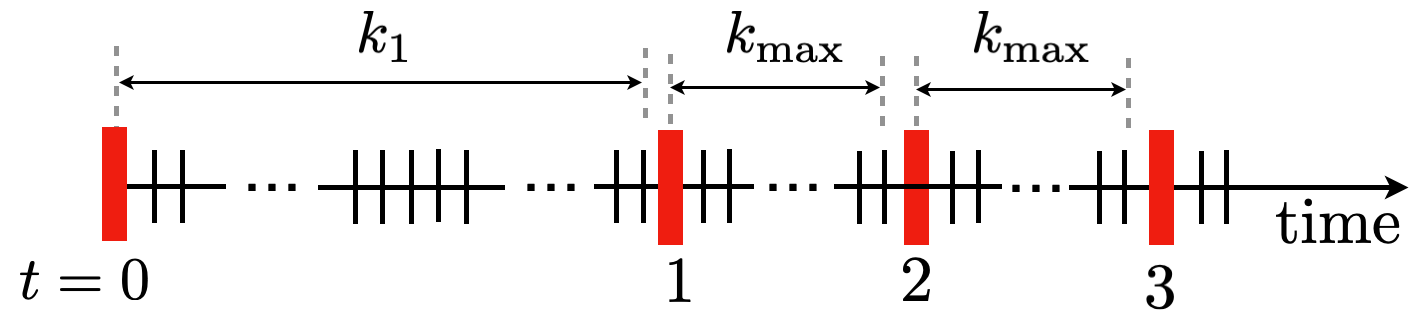}
		\caption{The finite-time consensus algorithm is terminated after $k_1=4(D_{\text{max}}+1)$ iterations	in the first step of the optimization. After the first  step, the consensus is terminated after $k_{\text{max}}=D_{\text{max}}+1$ iterations.}
		\label{ft} 
	\end{figure} 

{\color{black}\begin{remark}
Compared to conventional distributed optimization algorithms, such as AB \cite{xin2018linear} and Push-Pull \cite{pu2020push}, PrFTPS-GD requires additional communication rounds for each optimization step, as illustrated in Fig. \ref{ft}. Although the separate time-scales for optimization and consensus steps may slow down the convergence speed, they are crucial for ensuring privacy preservation and accuracy of PrFTPS-GD, as demonstrated by the rigorous theoretical analysis presented in Sections \ref{analysis} and \ref{pra}. \end{remark}}
	
\subsection{Convergence Analysis} \label{analysis}
		In this subsection, we provide the proof of the convergence and accuracy of Algorithm 1 and Algorithm 2. 
		
		From the weight mechanism, it can be seen that for $k\ge 1 $, the coupling weights are constants. Hence,  iteration \eqref{e1}   can  be written by using matrix-vector notation as follows: 
		\begin{equation}	 \label{e2}   
				\bm{x_l}(k+1)=\bm{\hat P}\bm{x_l}(k), \quad 
			\forall k\ge 1, l=1,2,
		\end{equation} 
		where $$\bm{x_l}(k)=[x_{1,l}^\alpha(k),\ldots,x_{n,l}^\alpha(k), x_{1,l}^\beta(k),\ldots,x_{n,l}^\beta(k)]^\top,$$
	$$	
			\bm{\hat P}=\left[
			\begin{array}{cc}
				\bm{P}&\frac{1}{2}\bm{I}_{n}\\ 
				\bm{\Lambda} & \frac{1}{2}\bm{I}_{n}
			\end{array}
			\right]
	$$
		with  $\bm{\Lambda}=\text{diag}(a_1^{\beta,\alpha},\ldots, a_n^{\beta,\alpha})$ and  $\bm{P}=[p_{ij}]$.
		
		Moreover, equation \eqref{eqa} can be rewritten as
		\begin{equation}\label{eq:op}
			\bm{x}(t+1)=\mathbf{A} \bm{x}(t)-\eta \bm{y}(t),
					\end{equation}
					where $					\mathbf{A}=\mathbf{\bar A} \otimes \mathbf{I}_p, \bm{x}(t)=[x_1(t)^\top,\ldots,x_n(t)^\top]^\top$  and $\bm{y}(t)=[y_1(t)^\top,\ldots,y_n(t)^\top]^\top.$
					
Before presenting Theorem 1, the following lemmas are needed.
			\begin{lemma}\label{lem1}(\hspace{-0.001cm}{\color{black}Theorem 8.4.4 in \cite{horn2012matrix}})
	 Under Assumption \ref{asp1}, the matrix $A$ has a unique nonnegative left eigenvector $u^\top$ (with respect to eigenvalue $1$) with $u^\top \mathbf{1}_{np}=np.$ 
\end{lemma}

\begin{lemma}\label{lemma4}(\hspace{-0.001cm}Adapted from Lemma 4 in \cite{pu2020push}) {\color{black}Under Assumptions \ref{asp1}, there exists a matrix  norm $||\cdot||_A$, defined as $||M||_A=AM A^{-1}$ for all $M \in \mathbb{R}^{np\times np}$, where $A\in \mathbb{R}^{{np}\times {np}}$  is invertible,    such that $\sigma_A:=||\mathbf{A}-\frac{\mathbf{1}_{np} u^\top}{n}||_A<1$, where $\mathbf{A}$ is the update matrix defined in \eqref{eq:op}, and $\sigma_A$ is arbitrarily  close to the spectral radius $\rho(\mathbf{A}-\frac{\mathbf{1}_{np} u^\top}{n})<1$.} 
\end{lemma}

Now, we  present Theorem 1 in the following.

						\begin{theorem}\label{t1}
			Under   Assumptions {\color{black}\ref{asp1} and \ref{asp3}}, for each node $j\in\mathcal{V}$, 
			
			1) Algorithm \ref{alg1} outputs the exact average of initial values of all nodes, i.e., $\forall j\in\mathcal{V}, \hat x_j^{ave}=\frac{1}{n}\sum_{i \in\mathcal{V}}x_i(0).$
			
		{\color{black}	2) When $0<\eta<\frac{1}{\mu+L}$, where $\mu, L$ are defined in Assumption  \ref{asp3}, Algorithm \ref{alg2} converges linearly related to
			the  optimization iteration number to the global optimal, i.e., $||\bm{x}(t)-\bm{1}\otimes x^\star||_2$ converges to $0$ linearly.}
		\end{theorem}

	\begin{proof}
		See Appendix \ref{pt1}.
	\end{proof}
					
\subsection{Privacy-preserving Performance Analysis}\label{pra}
		In this subsection, we analyze the privacy-preserving performance of   Algorithm \ref{alg2}  against honest-but-curious nodes and eavesdroppers. First, the   following assumption is needed.
		
		\begin{assumption}\label{asp3_1}
			Considering a  digraph  $\mathcal{G}(\mathcal{V}, \mathcal{E})$, each agent $i, \forall i \in \mathcal{V}$, does not know the structure of the whole network, i.e., the Laplacian of the network. 
		\end{assumption}
		
This assumption shows that 	agent $i$ has no access to the whole consensus dynamics in \eqref{e1}, which is very natural in distributed systems since agent $i$ is only aware of  its outgoing link weights $p_{ji}(k), j\in N_i^{-}\cup\{i\}$. Without other agents'  weights,   matrix $\bm{\hat P}$ in $\eqref{e2}$ is inaccessible to agent $i$.

		Note that  only local gradient information is exchanged in Algorithm \ref{alg1}, and  the outputs of Algorithm \ref{alg1} are the same for each agent. Hence, if Algorithm \ref{alg1} is able to preserve privacy of each agent in the network, we can deduce that Algorithm \ref{alg2} can preserve privacy of each agent.

%
%

Next, we show the privacy preservation  of Algorithm \ref{alg1}. 
		
		Under Algorithm \ref{alg1}, the information set accessible to the set of honest-but-curious nodes $\mathcal{N}$ at time $k$ can be defined as 
		 \begin{equation*}\label{info}
		\begin{aligned}
			\mathcal{I}_\mathcal{N}(k)= \{&x_{a,l}^\alpha(k), x_{a,l}^\beta(k),p_{ja}(k), p_{ap}(k), x_{p,l}^\alpha(k),\\&\mid p\in N_a^{-},a\in\mathcal{N}, j\in \mathcal{V}, l=1,2\}.
		\end{aligned}
		\end{equation*}

		Similar, an eavesdropper $\mathcal{R}$ is assumed to  eavesdrop some edges $\varepsilon_{ij}\in \mathcal{E}_{\mathcal{R}}$ and  its information set
		 is denoted by
	$$
			\mathcal{I}_\mathcal{R}(k)\triangleq \{ x_{j,l}^\alpha(k), p_{ij}(k)\mid \forall \varepsilon_{ij}\in \mathcal{E}_{\mathcal{R}}, j\in \mathcal{V}, l=1,2\}.
	$$

		\begin{theorem}\label{t2}
			Under Assumptions \ref{asp1} and \ref{asp3_1}, for each node $j \in\mathcal{V}$, under Algorithm \ref{alg1}, the privacy of agent $j$  can be preserved:
			
			1) Against a   set of honest-but-curious nodes $\mathcal{N}$  if at least one neighbor of node $j$ does not belongs to $\mathcal{N}$, i.e., $N_j^{+} \cup N_j^{-}\nsubseteq \mathcal{N}$. 
			
			2) Against eavesdropper $\mathcal{R}$ {\color{black}if there exists one edge $\varepsilon_{mj}$ or $\varepsilon_{jm}$ that eavesdropper $\mathcal{R}$ cannot eavesdrop, where $m\in N_j^+$ or $m\in N_j^-$.}
			\end{theorem}
						
			\begin{proof}
					See Appendix \ref{pt2}.			
			\end{proof}

		\section{SIMULATIONS}
		Consider a strongly connected digraph containing $n=5$ agents. The following distributed least squares problem is considered:
		
		\begin{equation}\label{pro2}
			\min\limits_{x\in\mathbb{R}^p} F(x) =\frac{1}{{\color{black}n}}\sum\limits_{i=1}^n f_i(x)=\frac{1}{5}\sum\limits_{i=1}^5 ||A_i x-b_i||^2,\end{equation}
		where $A_i \in\mathbb{R}^{q\times p}$ is only known to node $i$, $b_i \in \mathbb{R}^{q}$ is the measured data and $x \in \mathbb{R}^{p}$ is the common decision variable. In this simulation,  we set $q=p = 3$ and  $\eta=0.1$.  All elements of $A_i$ and $b_i$ are set from independent and identically distributed  samples of standard normal distribution $\mathcal{N}(0, 1)$. The finite-time consensus (Algorithm 1) stage consists of $k_1=64$ $(t=0)$ and $k_{\text{max}}=17$ $(t\ge 1)$  communication steps inside each PrFTPS-GD optimization iteration, i.e., the optimization  variable ${\bm x(0)}$ takes $64$ steps to become $\bm x(1)$ and then variable ${\bm x(t)}$ is updated every $17$ iterations for $t\ge1$.
	\begin{figure}[t]
			\centering
			\includegraphics[width=0.5\textwidth]{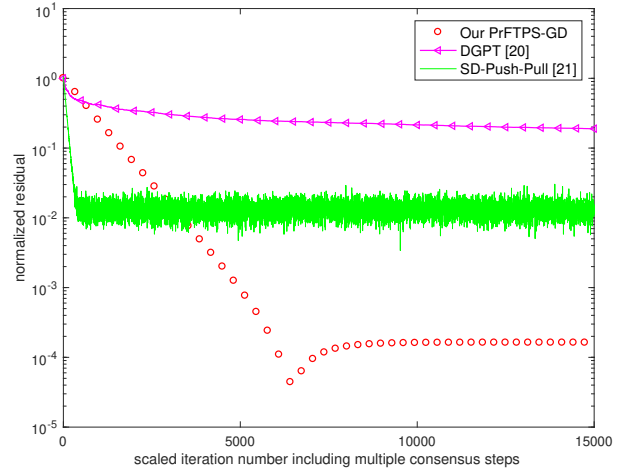}
			\caption{Performance comparison among  our proposed PrFTPS-GD,    DPGT \cite{wang2022tailoring} and  SD-Push-Pull \cite{chen2021differential}. }
			\label{comp}
		\end{figure}

		The normalized residual $\sum_i^5(||x_i(t)-x^\star||/||x_i(0)-x^\star||)$ is illustrated in  Fig. \ref{comp} to compare  PrFTPS-GD with DPGT\cite{wang2022tailoring} and  SD-Push-Pull \cite{chen2021differential}. Notice that as we have multiple consensus steps in Algorithm \ref{alg1} inside our PrFTPS-GD while there is only one step in DPGT  and  SD-Push-Pull, we have scaled each PrFTPS-GD optimization iteration number to include the consensus number (i.e., $k_1$ and $k_{\text{max}}$) directly.  It is shown that the proposed PrFTPS-GD converges linearly related to optimization iteration number.  The stepsizes of all algorithms are manually tuned to obtain the corresponding best convergence performance. For SD-Push-Pull,  the  stepsize is set to be $\eta=0.1$ and it can be seen that SD-Push-Pull only converge to suboptimality while PrFTPS-GD can converge to the optimal point. In terms of DPGT, we  choose the  stepsize with the diminishing sequence as $$
			\lambda^k=\frac{0.02}{1+0.1k}, \quad \gamma_1^k=\frac{1}{1+0.1k^{0.9}}, \quad \gamma_2^k=\frac{1}{1+0.1k^{0.8}}.
		$$ Fig. \ref{comp} demonstrates clearly   that  PrFTPS-GD converges faster to the optimal solution than DPGT.

		\section {CONCLUSION AND FUTURE WORK}
		In this paper, a privacy-preserving  finite-time push-sum based gradient descent algorithm is proposed to solve  the distributed optimization problem over a directed graph. Compared to  existing privacy-preserving algorithms in the literature, the proposed one can converge  linearly  to the global optimum. Moreover, privacy of each agent is preserved via a state decomposition mechanism.  
		
		Future work includes   considering constrained optimization problems in  large-scale. Moreover, privacy-preserving distributed optimiation algorithm with quantization communication is a potential research direction.

%
\appendix
 	\subsection{Proof of Theorem \ref{t1}}\label{pt1}
 	
		1) By the weight mechanism  and the iteration in \eqref{e1} at iteration $k=0$, we have $\bm{1}^\top \bm{x_1}(1)=\bm{1}^\top \bm{x_1}(0)=2\sum_{i=1}^n x_i(0)$ and
					$\bm{1}^\top \bm{x_2}(1)=\bm{1}^\top \bm{x_2}(0)=2n. $


			Since $\bm{\hat P}$ is irreducible, column-stochastic with positive diagonals, from  Perron-Frobenius theorem, we have $\rho(\bm{\hat P})=1$. Denote $\bm{v}=[v_i]$ as the right eigenvector corresponding to the eigenvalue of $1$, we have $\lim_{k\rightarrow \infty}\bm{\hat P}^k=\bm{v1}^\top$. 
			Then,  
			\begin{equation*} 
\begin{aligned}\label{cov1}
					\frac{x^\alpha_{j,1}(\infty)}{ x^\alpha_{j,2}(\infty)} &=\frac{[\bm{\hat P}^{\infty}\bm{x_1}({ 1})]_j}{[\bm{\hat P}^{\infty}\bm{x_2}({ 1})]_j}=\frac{[\bm{v1}^\top\bm{x_1}({ 1})]_j}{[\bm{v1}^\top\bm{x_2}({ 1})]_j}=\frac{\sum_{i=1}^n x_i(0)}{n}.\\	
				\end{aligned}
			\end{equation*}	 
			Since the coefficient $\bm{\beta}_j$ is independent of the initial node state, 
			\begin{equation}
				\hat x_j^{ave}=\frac{(x_{1,{D_j}}^\alpha)^\top \bm{\beta}_j}{(x_{2,{D_j}}^\alpha)^\top \bm{\beta}_j}=\frac{x^\alpha_{j,1}(\infty)}{ x^\alpha_{j,2}(\infty)} =\frac{\sum_{i=1}^n x_i(0)}{n} .
			\end{equation} 
			
			{\color{black}2) Denote $\bar x(t)=u^\top \bm{x}(t)/n, \bar y(t)=\frac{1}{n}\sum_{i=1}^n \nabla f_i( x_{i}(t))$ and $g(t)=\frac{1}{n}\sum_{i=1}^n\nabla f_i(\bar x(t))$.
			From the above analysis, we know that at time iteration $t$, each agent can obtain the average gradient at time $t$ via Algorithm \ref{alg1}, i.e, $y_i(t)=\bar y(t), \bm{y}(t)=\bm{1} \bar y(t)$. Hence, from iteration \eqref{eq:op}, we have
			\begin{equation*}
			\begin{aligned}
				\bar x(t+1)-x^\star&=\bar x(t)-\eta \bar y(t)-x^\star\\
&=\bar x(t)-\eta g(t)-x^\star-\eta(\bar y(t)-g(t)),\\
			\end{aligned}
							\end{equation*}
		\begin{equation*}
			\begin{aligned}
\bm{x}(t+1)-\bm{1}\bar x(t+1)&=\mathbf{A}\bm{x}(t)-\eta \bm{y}(t)-\bm{1}\bar x(t)+\eta \bm{1}\bar y(t)\\
&=(\mathbf{A}-\mathbf{1}_n u^\top/n)(\bm{x}(t)-\bm{1}\bar x(t)).\\
			\end{aligned}
							\end{equation*}

		{\color{black}Based on [30, Lemma 8(c) and 10]}, if $\eta< 1/(\mu+L)$, we have 
	\begin{equation}
		\begin{aligned}
			&||\bar x(t+1)-x^\star||_2  \le (1-\eta\mu)||\bar x(t)-x^\star||_2+\eta||\bar y(t)-g(t)||\\
			&\le (1-\eta\mu)||\bar x(t)-x^\star||_2+\frac{\eta L}{\sqrt{n}}||\bm{x}(t)-\bm{1}\bar x(t)||_2.\\
		\end{aligned}
	\end{equation}
{\color{black}Moreover, from the result of Lemma \ref{lemma4}, we can obtain that 
$$||\bar x(t+1)-x^\star||_2  \le  (1-\eta\mu)||\bar x(t)-x^\star||_2+\frac{\eta qL}{\sqrt{n}}||\bm{x}(t)-\bm{1}\bar x(t)||_A.
$$}
	
	{\color{black}Denote $V(t)=[||\bar x(t+1)-x^\star||_2,  ||\bm{x}(t)-\bm{1}\bar x(t)||_A]^\top$}, we have 
	\begin{equation}
		V(t+1)\le \bm{M} V(t),
	\end{equation}
	where the transition matrix $\bm{M}=\begin{bmatrix}
		1-\eta\mu & {\color{black}\eta} q L /\sqrt{n}\\
		0 & \sigma_A\\
	\end{bmatrix}$.
	
	Since $0<1-\eta\mu<1$ and $0<\sigma_A<1$, we can obtain that the spectral radius of $\bm{M}$ is strictly less than $1$ and therefore $||\bm{x}(t)-\bm{1}\otimes x^\star||_2$ converges to zero linearly.	}

	 	\subsection{Proof of Theorem \ref{t2}}\label{pt2}
	
	Since under Algorithm \ref{alg1}, the maximum communication round  equals to $k_1$, the information set $\mathcal{I}_\mathcal{N}(0:k_1)$ and $\mathcal{I}_\mathcal{R}(0:k_1)$ denote all the information accessible to the adversary. {\color{black}From  Algorithm \ref{alg2}, it can be seen that the private information $\nabla f_j(x_j(t)), \forall t\ge 0$ is regarded as the input of Algorithm \ref{alg1}. Hence, it suffices to prove that the privacy of initial value of agent $j$,   $x_j(0)$ is preserved under Algorithm \ref{alg1}. }
				
				1) Denote $\mathcal{L}=\mathcal{V}\backslash \mathcal{A}$ as the set of legitimate nodes. Since $N_j^{+} \cup N_j^{-}\nsubseteq \mathcal{A}$, there exists at least one node $m$ that belongs to $N_j^{+} \cup N_j^{-}$ but not $\mathcal{A}$. Fix any  feasible information set $\mathcal{I}_\mathcal{N}(0:k_1)$. We denote {\color{black} $\{x_{n,1}^\alpha(0)', x_{n,1}^\beta(0)',p_{in}(0)'\mid n\in \mathcal{L}, i\in \mathcal{V}\}$ }as an arbitrary set of initial substate values and  weights that satisfies $\mathcal{I}_\mathcal{N}(0:k_1)$. Hence, we have  $x_j(0)'=(x_{j,1}^\alpha(0)'+x_{j,1}^\beta(0)')/2.$ We then denote $x_j(0)''$ as $x_j(0)''=x_i(0)'+e$, where $e$ is an arbitrary real number.
				Next we show that there exists a set of values $\{x_{n,1}^\alpha(0)'', x_{n,1}^\beta(0)'', p_{in}(0)'', n \in \mathcal{L}, i\in \mathcal{V}\}$ which makes $x_j(0)''\in \Delta_j(\mathcal{I}_\mathcal{N}(0:k_1)).$ 
				The initial substate values $x_{n,1}^\alpha(0)'', x_{n,1}^\beta(0)''$ are denoted as follows.
			
				\begin{equation}\label{sub}   
					\begin{aligned}		
					&x_{q,1}^\alpha(0)''= x_{q,1}^\alpha(0)', x_{q,1}^\beta(0)''= x_{q,1}^\beta(0)',\forall q\in  \mathcal{L} \backslash {\color{black}\{j,m\}},\\
						&x_{m,1}^\alpha(0)''= x_{m,1}^\alpha(0)', x_{m,1}^\beta(0)''= x_{m,1}^\beta(0)'-2e,\\
									&x_{j,1}^\alpha(0)''= x_{j,1}^\alpha(0)',  x_{j,1}^\beta(0)''= x_{j,1}^\beta(0)'+2e.\\	\end{aligned}   		
				\end{equation}
				Then we consider the following two situations.
				
				Situation \uppercase\expandafter{\romannumeral 1}: Consider $m\in N^{-}_j$,  {\color{black}then the information set sequence accessible to set $\mathcal{N}$ equals to $\mathcal{I}_\mathcal{N}(0:k_1)$  under the initial substate values in (\ref{sub}) and the following weights:}
				
				\begin{equation}\label{e4}		
					\begin{aligned}
						&a_n^{(\beta,\alpha)} (0)''=a_n^{(\beta,\alpha)} (0)',\forall n\in \mathcal{L}, \\
						& p_{mm}(0)''=(p_{mm}(0)'x_{m,1}^\alpha(0)'+2e)/x_{m,1}^\alpha(0)',\\
						&p_{jm}(0)''=(p_{jm}(0)'x_{m,1}^\alpha(0)'-2e)/x_{m,1}^\alpha(0)',\\
						&p_{qm}(0)''= p_{qm}(0)' , \forall q\in \mathcal{L}\backslash{\color{black}\{j,m\}}, \\
						&p_{np}(0)''=p_{np}(0)', \forall n\in \mathcal{L},\forall p \in \mathcal{L}\backslash\{m\}.
					\end{aligned}	
				\end{equation}		
				
				Situation \uppercase\expandafter{\romannumeral 2}: Consider $m\in N^{+}_j$, then the information set sequence accessible to set $\mathcal{N}$ equals to $\mathcal{I}_\mathcal{N}(0:k_1)$ under the initial substate values in (\ref{sub}) and the following weights:
			
				\begin{equation}\label{e5}
					\begin{aligned}
						&a_n^{(\beta,\alpha)} (0)''=a_n^{(\beta,\alpha)} (0)',\forall n\in \mathcal{L}, \\
						&p_{jj}(0)''=(p_{jj}(0)'x_{j,1}^\alpha(0)'-2e)/x_{j,1}^\alpha(0)',\\
						& p_{mj}(0)''=(p_{mj}(0)'x_{j,1}^\alpha(0)'+2e)/x_{j,1}^\alpha(0)',\\
						&  p_{qj}(0)''= p_{qj}(0)' , \forall q\in \mathcal{L}\backslash{\color{black}\{j,m\}}, \\
						&p_{np}(0)''=p_{np}(0)', \forall n\in \mathcal{L},\forall p \in \mathcal{L}\backslash\{j\}.
					\end{aligned}	
				\end{equation}		
				
				Summarizing Situations \uppercase\expandafter{\romannumeral 1} and \uppercase\expandafter{\romannumeral 2}, we have that $ x_j(0)''=x_j(0)'+e\in \Delta_j(\mathcal{I}_\mathcal{N}(0:k_1))$, then	
				$$\text{Diam}(\Delta_j(\mathcal{I}_\mathcal{N}(0:k_1)))\geq \sup\limits_{e\in\mathbb{R}} |x_j(0)'-(x_j(0)'+e)| =\infty.	  $$ 
				From Definition \ref{dfprivacy}, the first statement is  proved. 
				
				2) For the second statement, due to the topology constraints,   {\color{black}the eavesdropper $\mathcal{R}$ can not eavesdrop on $\varepsilon_{mj}$ or $\varepsilon_{jm}$, where $m\in N_j^+$ or $m\in N_j^-$, i.e., $p_{mj}(0)$ or $p_{jm}(0)$ is inaccessible to $\mathcal{R}$. Moreover,  since the self-weights $p_{jj}(0), p_{mm}(0)$ are not transmitted over the communication network, the eavesdropper $\mathcal{R}$ can not  obtain any information of $p_{jm}(0),  p_{mm}(0)$ or $p_{mj}(0), p_{jj}(0))$.} Then, similar to the proof in the first statement, we have $\text{Diam}(\Delta_j(\mathcal{I}_\mathcal{R}(0:k_1)))=\infty$, i.e., the privacy of $x_j(0)$ is preserved. 
\balance

	\end{document}